\documentclass{amsart}
\usepackage{graphicx, color}
\usepackage{amscd}
\usepackage{amsmath,empheq}
\usepackage{amsfonts}
\usepackage{amssymb}
\usepackage{mathrsfs}
\usepackage[all]{xy}
\newtheorem{theorem}{Theorem}

\newtheorem{lemma}[theorem]{Lemma}
\newtheorem{prop}[theorem]{Proposition}
\newtheorem{remark}{Remark}

\newtheorem{claim}{Claim}

\newenvironment{proof-sketch}{\noindent{\bf Sketch of Proof}\hspace*{1em}}{\qed\bigskip}
 
\everymath{\displaystyle} 
\newcommand{\RR}{\mathbb R}
\newcommand{\NN}{\mathbb N}

\newcommand{\ZZ}{\mathbb Z}
\newcommand{\bb}{\begin{equation}}
\newcommand{\bbb}{\end{equation}}
\newcommand{\intom}{\int_\Omega}
\renewcommand{\le}{\leqslant}
\renewcommand{\leq}{\leqslant}
\renewcommand{\ge}{\geqslant}
\renewcommand{\geq}{\geqslant}
\begin{document}
\title[Double phase problems with reaction of arbitrary growth]{Double-phase problems with reaction of arbitrary growth}
\author[N.S. Papageorgiou]{Nikolaos S. Papageorgiou}
\address[N.S. Papageorgiou]{National Technical University, Department of Mathematics,
				Zografou Campus, Athens 15780, Greece \&  Institute of Mathematics, Physics and Mechanics, Jadranska 19, 1000 Ljubljana, Slovenia}
\email{\tt npapg@math.ntua.gr}
\author[V.D. R\u{a}dulescu]{Vicen\c{t}iu D. R\u{a}dulescu}
\address[V.D. R\u{a}dulescu]{Faculty of Applied Mathematics, AGH University of Science and Technology, al. Mickiewicza 30, 30-059 Krak\'ow, Poland \& Institute of Mathematics ``Simion Stoilow" of the Romanian Academy, P.O. Box 1-764, 014700 Bucharest, Romania \&  Institute of Mathematics, Physics and Mechanics, Jadranska 19, 1000 Ljubljana, Slovenia}
\email{\tt vicentiu.radulescu@imar.ro}
\author[D.D. Repov\v{s}]{Du\v{s}an D. Repov\v{s}}
\address[D.D. Repov\v{s}]{Faculty of Education and Faculty of Mathematics and Physics, University of Ljubljana, \&  Institute of Mathematics, Physics and Mechanics, Jadranska 19, 1000 Ljubljana, Slovenia}
\email{\tt dusan.repovs@guest.arnes.si}
\begin{abstract}
We consider a parametric nonlinear nonhomogeneous elliptic equation, driven by the sum of two differential operators having different structure. The associated energy functional has unbalanced growth and we do not impose any global growth conditions to the reaction term, whose behavior is prescribed only near the origin. Using  truncation and comparison techniques and Morse theory, we show that
the problem has multiple solutions in the case of high perturbations.
We also show that if a symmetry condition is imposed to the reaction term, then we can generate a sequence of distinct nodal solutions with smaller and smaller energies.
\end{abstract}
\maketitle

\section{Introduction}
This paper was motivated by several recent contributions to the qualitative analysis of nonlinear problems with unbalanced growth. We first refer to the pioneering contributions of Marcellini \cite{marce1,marce2} who studied lower semicontinuity and regularity properties of minimizers of certain quasiconvex integrals. Problems of this type arise in nonlinear elasticity and are connected with the deformation of an elastic body, cf. Ball \cite{ball1,ball2}.

In order to recall the roots of double phase problems, let us assume that $\Omega$ is a bounded domain in $\RR^N$ ($N\geq 2$) with smooth boundary. If $u:\Omega\to\RR^N$ is the displacement and if $Du$ is the $N\times N$  matrix of the deformation gradient, then the total energy can be represented by an integral of the type
\bb\label{paolo}I(u)=\intom F(x,Du(x))dx,\bbb
where the energy function $F=F(x,\xi):\Omega\times\RR^{N\times N}\to\RR$ is quasiconvex with respect to $\xi$. One of the simplest examples considered by Ball is given by functions $F$ of the type
$$F(\xi)=g(\xi)+h({\rm det}\,\xi),$$
where ${\rm det}\,\xi$ is the determinant of the $N\times N$ matrix $\xi$, and $g$, $h$ are nonnegative convex functions, which satisfy the growth conditions
$$g(\xi)\geq c_1\,|\xi|^p;\quad\lim_{t\to+\infty}h(t)=+\infty,$$
where $c_1$ is a positive constant and $1<p<N$. The condition $p\leq N$ is necessary to study the existence of equilibrium solutions with cavities, that is, minima of the integral \eqref{paolo} that are discontinuous at one point where a cavity forms; in fact, every $u$ with finite energy belongs to the Sobolev space $W^{1,p}(\Omega,\RR^N)$, and thus it is a continuous function if $p>N$.

In accordance with these problems arising in nonlinear elasticity, Marcellini \cite{marce1,marce2} considered continuous functions $F=F(x,u)$ with {\it unbalanced growth} that satisfy
$$c_1\,|u|^p\leq |F(x,u)|\leq c_2\,(1+|u|^q)\quad\mbox{for all}\ (x,u)\in\Omega\times\RR,$$
where $c_1$, $c_2$ are positive constants and $1< p< q$. Regularity and existence of solutions of elliptic equations with $p,q$--growth conditions were studied in \cite{marce2}.

The study of non-autonomous functionals characterized by the fact that the energy density changes its ellipticity and growth properties according to the point has been continued in a series of remarkable papers by Mingione {\it et al.} \cite{mingi1,mingi2,mingi3,mingi4,mingi5}. These contributions are in relationship with the works of Zhikov \cite{zhikov1}, in order to describe the
behavior of phenomena arising in nonlinear
elasticity.
In fact, Zhikov intended to provide models for strongly anisotropic materials in the context of homogenisation.
We also point out that these functionals revealed to be important in the study of duality theory
and in the context of the Lavrentiev phenomenon \cite{zhikov2}.
One of the problems considered by Zhikov was the {\it double phase} functional
$$
{\mathcal P}_{p,q}(u) :=\intom (|Du|^p+a(x)|Du^q)dx,\quad 0\leq a(x)\leq L,\ 1<p<q.
$$
where the modulating coefficient $a(x)$ dictates the geometry of the composite made by
two differential materials, with hardening exponents $p$ and $q$, respectively.

Motivated by these results, we study in this paper a paper with $(p,2)$--growth. More precisely, we consider the following nonlinear, nonhomogeneous parametric Dirichlet problem
\begin{equation}
-\Delta_pu(z)-\Delta u(z)=\lambda f(z,u(z))\ \mbox{in}\ \Omega,\ u|_{\partial\Omega}=0,\ 2<p<\infty,\ \lambda>0,\tag{$P_{\lambda}$}\label{eqP}
\end{equation}
where $\Omega\subseteq\RR^N$ is a bounded domain with smooth $C^2$-boundary $\partial\Omega$.

For $q\in(1,\infty)$, we denote by $\Delta_q$  the $q$-Laplace differential operator defined by
$$\Delta_qu={\rm div}\,(|Du|^{q-2}Du)\ \mbox{for all}\ u\in W^{1,q}_{0}(\Omega).$$

If $q=2$, then $\Delta_2=\Delta$ is the usual Laplacian.

So, in problem \eqref{eqP} the differential operator (the left-hand side of the equation), is not homogeneous. The reaction term $f(z,x)$ is a Carath\'eodory function (that is, for all $x\in\RR$ the mapping $z\mapsto f(z,x)$ is measurable and for almost all $z\in\Omega$, $x\mapsto f(z,x)$ is continuous). Here the interesting feature of our work is that no global growth conditions are imposed on $f(z,\cdot)$. Instead, all our hypotheses on $f(z,\cdot)$ concern its behavior near zero. Our goal is to show that under these minimal conditions on the reaction term, we can obtain multiplicity results for problem \eqref{eqP} when the parameter $\lambda>0$ is big enough. Moreover, we provide sign information for all solutions we produce. Using variational methods combined with truncation and comparison techniques and Morse theory, we prove two multiplicity theorems, producing respectively three and four nontrivial smooth solutions, all with sign information. When a symmetry condition is imposed on $f(z,\cdot)$ (namely, that $f(z,\cdot)$ is odd) we show that we can have an entire sequence of smooth nodal (that is, sign changing) solutions converging to zero in $C^1_0(\overline{\Omega})$.

Recently, multiplicity theorems with sign information for the solutions of $(p,2)$-equations (that is, equations driven by the sum of a $p$-Laplacian and a Laplacian), have been proved by Aizicovici, Papageorgiou and Staicu \cite{2}, Papageorgiou and R\u adulescu \cite{17, 18}, Papageorgiou, R\u adulescu and Repov\v{s} \cite{20}, Papageorgiou and Smyrlis \cite{21}, Sun \cite{23} and Sun, Zhang and Su \cite{24}. In all these works, it is assumed that the reaction term has subcritical polynomial growth. We mention that $(p,2)$-equations arise in problem of mathematical physics, see Cherfils and Ilyasov \cite{5} (reaction diffusion equations), Derrick \cite{7} (elementary particles), and Wilhelmsson \cite{25} (plasma physics).

\section{Mathematical Background}
Let $X$ be a Banach space and let $X^*$ be its topological dual. We denote by $\left\langle \cdot,\cdot\right\rangle$ the duality brackets for the pair $(X,X^*)$. Given $\varphi\in C^1(X,\RR)$, we say that $\varphi$ satisfies the ``Palais-Smale condition" (the ``PS-condition" for short), if the following property holds:

\smallskip
``Every sequence $\{u_n\}_{n\geq 1}\subseteq X$ such that $\{\varphi(u_n)\}_{n\geq 1}\subseteq\RR$ is bounded and
$$\varphi'(u_n)\rightarrow 0\ \mbox{in}\ X^*\ \mbox{as}\ n\rightarrow\infty,$$
\indent admits a strongly convergent subsequence".

\smallskip
This is a compactness-type condition on the functional $\varphi$ and leads to a deformation theorem from which one can derive the minimax theory of the critical values of $\varphi$. One of the main results in this theory is the ``mountain pass theorem" of Ambrosetti and Rabinowitz \cite{3}.
\begin{theorem}\label{th1}
	Let $X$ be a Banach space. Assume that $\varphi\in C^1(X,\RR)$ satisfies the PS-condition, $u_0,u_1\in X,\ ||u_1-u_0||>r$,
	$$\max\{\varphi(u_0),\varphi(u_1)\}<\inf[\varphi(u):||u-u_0||=\rho]=m_{\rho}.$$
	Set $c={\underset{\mathrm{\gamma\in\Gamma}}\inf}{\underset{\mathrm{0\leq t\leq 1}}\max}\varphi(\gamma(t))$, where $\Gamma=\{\gamma\in C([0,1],X):\gamma(0)=u_0,\gamma(1)=u_1\}$. Then $c\geq m_{\rho}$ and $c$ is a critical value of $\varphi$ (that is, there exists $\hat{u}\in X$ such that $\varphi'(\hat{u})=0$, $\varphi(\hat{u})=c$).
\end{theorem}

In the analysis of problem \eqref{eqP} we will use the Sobolev space $W^{1,p}_{0}(\Omega)$ and the ordered Banach space
$$C^1_0(\overline{\Omega})=\{u\in C^1(\overline{\Omega}):u|_{\partial\Omega}=0\}.$$

By $||\cdot||$ we denote the norm of $W^{1,p}_{0}(\Omega)$. Using the Poincar\'e inequality we can say that
$$||u||=||Du||_p\ \mbox{for all}\ u\in W^{1,p}_{0}(\Omega).$$

The positive (order) cone of $C^1_0(\overline{\Omega})$ is
$$C_+=\{u\in C^1_0(\overline{\Omega}):u(z)\geq 0\ \mbox{for all}\ z\in\overline{\Omega}\}.$$

This cone has a nonempty interior given by
$$D_+=\{u\in C_+:u(z)>0\ \mbox{for all}\ z\in\Omega,\left.\frac{\partial u}{\partial n}\right|_{\partial\Omega}<0\},$$
 where $n(\cdot)$ denotes the outward unit normal on $\partial\Omega$.

Suppose that $\hat{f}:\Omega\times\RR\rightarrow\RR$ is a Carath\'eodory function satisfying
$$|\hat{f}(z,x)|\leq a(z)(1+|x|^{r-1})\ \mbox{for almost all}\ z\in\Omega,\ \mbox{and all}\ x\in\RR,$$
with $a\in L^{\infty}(\Omega)_+$, $1<r\leq p^*$, where $p^*=\left\{\begin{array}{ll}
\frac{Np}{N-p}&\mbox{if}\ p<N\\
+\infty&\mbox{if}\ N\leq p
\end{array}\right.$ (the Sobolev critical exponent). We set $\hat{F}(z,x)=\int^x_0\hat{f}(z,s)ds$ and consider the $C^1$-functional $\hat{\varphi}:W^{1,p}_{0}(\Omega)\rightarrow\RR$ defined by
$$\hat{\varphi}(u)=\frac{1}{p}||Du||^p_p+\frac{1}{2}||Du||^2_2-\int_{\Omega}\hat{F}(z,u(z))dz\ \mbox{for all}\ u\in W^{1,p}_{0}(\Omega).$$

The next result is essentially an outgrowth of the nonlinear regularity theory (see Lieberman \cite{15}) and can be found in Gasinski and Papageorgiou \cite{11} (the subcritical case) and Papageorgiou and R\u adulescu \cite{19} (the critical case).
\begin{prop}\label{prop2}
	Let $\hat{u}\in W^{1,p}_{0}(\Omega)$ be a local $C^1_0(\overline{\Omega})$-minimizer of $\hat{\varphi}$, that is, there exists $\rho_0>0$ such that
		$$\hat{\varphi}(\hat{u})\leq\hat{\varphi}(\hat{u}+h)\ \mbox{for all}\ h\in C^1_0(\overline{\Omega})\ \mbox{with}\ ||h||_{C^1_0(\overline{\Omega})}\leq\rho_0.$$
		Then $\hat{u}\in C^{1,\alpha}_0(\overline{\Omega})$ for some $\alpha\in(0,1)$ and $\hat{u}$ is also a local $W^{1,p}_{0}(\Omega)$-minimizer of $\hat{\varphi}$, that is, there exists $\rho_1>0$ such that
		$$\hat{\varphi}(\hat{u})\leq\hat{\varphi}(\hat{u}+h)\ \mbox{for all}\ h\in W^{1,p}_{0}(\Omega)\ \mbox{with}\ ||h||\leq\rho_1.$$
\end{prop}

As we have already mentioned in the introduction, some of our tools in the analysis of problem \eqref{eqP}, are the comparison results for such equations. One such result can be found in Filippakis, O'Regan and Papageorgiou \cite{8} and is an extension of a result for $p$-Laplacian equations due to Arcoya and Ruiz \cite[Proposition 2.6]{4}. First, let us introduce some notations. Given $h_1,h_2\in L^{\infty}(\Omega)$, we say that $h_1\prec h_2$ if and only if for every compact $K\subseteq \Omega$, we can find $\epsilon=\epsilon(K)>0$ such that
$$h_1(z)+\epsilon\leq h_2(z)\ \mbox{for almost all}\ z\in K.$$

Note that if $h_1,h_2\in C(\Omega)$ and $h_1(z)<h_2(z)$ for all $z\in\Omega$, then $h_1\prec h_2$.
\begin{prop}\label{prop3}
 If $\hat{\xi}\geq 0,\ h_1,h_2\in L^{\infty}(\Omega)$ with $h_1\prec h_2$ and $u\in C^1_0(\overline{\Omega})$, $v\in D_+$ satisfy
\begin{eqnarray*}
	&&-\Delta_pu-\Delta u+\hat{\xi}|u|^{p-2}u=h_1\ \mbox{in}\ \Omega,\\
	&&-\Delta_pv-\Delta v+\hat{\xi}v^{p-1}=h_2\ \mbox{in}\ \Omega,
\end{eqnarray*}
then $v-u\in D_+$.
\end{prop}

To produce a sequence of distinct smooth nodal solution, we will use an abstract result of Kajikiya \cite{14}, which is an extension of the symmetric mountain pass theorem.
\begin{theorem}\label{th4}
	Let $X$ be a Banach space. Assume that $\varphi\in C^1(X,\RR)$ satisfies the PS-condition, is even, bounded below, $\varphi(0)=0$ and for every $n\in\NN$ there exist an $n$-dimensional subspace $V_n$ of $X$ and $\rho_n>0$ such that
	$$\sup[\varphi(u):u\in V_n,||u||=\rho_n]<0.$$
	Then there exists a sequence $\{u_n\}_{n\geq 1}$ of critical points of $\varphi$ (that is, $\varphi'(u_n)=0$ for all $n\in\NN$) such that $u_n\rightarrow 0$ in $X$.
\end{theorem}

For $q\in(1,\infty)$, let $A_q:W^{1,q}_{0}(\Omega)\rightarrow W^{-1,q'}(\Omega)^*=W^{1,q}_{0}(\Omega)^*$ (with $\frac{1}{q}+\frac{1}{q'}=1$) be the nonlinear map defined by
$$\left\langle A_q(u),h\right\rangle=\int_{\Omega}|Du|^{q-2}(Du,Dh)_{\RR^N}dz\ \mbox{for all}\ u,h\in W^{1,q}_{0}(\Omega).$$

When $q=2$, we set $A_2=A$ and we have $A\in\mathcal{L}(H^1_0(\Omega),H^{-1}(\Omega))$. For the general map $A_q$ we have the following result summarizing its properties (see Gasinski and Papageorgiou \cite[p. 746]{10}).
\begin{prop}\label{prop5}
	The map $A_q:W^{1,q}_{0}(\Omega)\rightarrow W^{-1,q'}(\Omega)$ is strictly monotone, continuous (hence maximal monotone, too) and of type $(S)_+$, that is, if $u_n\stackrel{w}{\rightarrow}u$ in $W^{1,q}_{0}(\Omega)$ and $\limsup\limits_{n\rightarrow\infty}\left\langle A_q(u_n),u_n-u\right\rangle\leq 0$, then $u_n\rightarrow u$ in $W^{1,q}_{0}(\Omega)$.
\end{prop}

Another tool that we use in the analysis of problem \eqref{eqP} is the Morse theory (critical groups). So, let $X$ be a Banach space and $(Y_1,Y_2)$ a topological pair such that $Y_2\subseteq Y_1\subseteq X$. For every $k\in\NN_0$, we denote by $H_k(Y_1,Y_2)$  the $k$th singular homology group with integer coefficients for the pair $(Y_1,Y_2)$. Recall that for $k\in-\NN$ we have $H_k(Y_1,Y_2)=0$. Suppose that $\varphi\in C^1(X,\RR)$ and $c\in\RR$. We introduce the following sets:
\begin{eqnarray*}
	&&K_{\varphi}=\{u\in X:\varphi'(u)=0\},\\
	&&K^c_{\varphi}=\{u\in K_{\varphi}:\varphi(u)=c\},\\
	&&\varphi^c=\{u\in X:\varphi(u)\leq c\}.
\end{eqnarray*}

Let $u\in K^c_{\varphi}$ be isolated. The critical groups of $\varphi$ at $u$ are defined by
$$C_k(\varphi,u)=H_k(\varphi^c\cap U,\varphi^c\cap U\backslash\{u\})\ \mbox{for all}\ k\in\NN_0,$$
where $U$ is a neighborhood of $u$ such that $K_{\varphi}\cap\varphi^c\cap U=\{u\}$. The excision property of singular homology implies that the above definition is independent of the choice of the neighborhood of $U$.

Assume that $\varphi\in C^1(X,\RR)$ satisfies the PS-condition and $\inf\varphi(K_{\varphi})>-\infty$. Let $c<\inf\varphi(K_{\varphi})$. The critical groups of $\varphi$ at infinity are defined by
$$C_k(\varphi,\infty)=H_k(X,\varphi^c)\ \mbox{for all}\ k\in\NN_0.$$

The definition is independent of the choice of the level $c<\inf\varphi(K_{\varphi})$. Indeed, let $c'<\inf \varphi(K_{\varphi})$ be another such level. We assume that $c'<c$. Then we know that $\varphi^{c'}$ is a strong deformation retract of $\varphi^c$ (see Gasinski and Papageorgiou \cite[p. 628]{10}). So, we have
$$H_k(X,\varphi^c)=H_k(X,\varphi^{c'})\ \mbox{for all}\ k\in\NN_0$$
(see Motreanu, Motreanu and Papageorgiou \cite[p. 145]{16}).

Suppose that $K_{\varphi}$ is finite. We introduce the following formal series
\begin{eqnarray*}
	&&M(t,u)=\sum_{k\geq 0}{\rm rank}\, C_k(\varphi,u)t^k\ \mbox{for all}\ t\in\RR,\ \mbox{all}\ u\in K_{\varphi}\\
	&\mbox{and}&P(t,\infty)=\sum_{k\geq 0} {\rm rank}\, C_k(\varphi,\infty)t^k\ \mbox{for all}\ t\in\RR.
\end{eqnarray*}

These quantities are related via the Morse relation, which says that

\begin{eqnarray}\label{eq1}
	\sum_{u\in K_{\varphi}}M(t,u)=P(t,\infty)+(1+t)Q(t)\ \mbox{for all}\ t\in\RR,
\end{eqnarray}
with $Q(t)=\sum_{k\geq 0}\beta_kt^k$ being a formal series in $t\in\RR$ with nonnegative integer coefficients $\beta_k$.

Finally, let us introduce some basic notations that we will use in the sequel. Given $x\in\RR$, we set $x^{\pm}=\max\{\pm x,0\}$. Then for $u\in W^{1,p}_{0}(\Omega)$, we define
$$u^{\pm}(\cdot)=u(\cdot)^{\pm}.$$

We know that
$$u^{\pm}\in W^{1,p}_{0}(\Omega),u=u^+-u^-\ \mbox{and}\ |u|=u^++u^-.$$

For  a measurable function $g:\Omega\times\RR\rightarrow\RR$ (for example, for a Carath\'eodory function $g(\cdot,\cdot)$), we denote by $N_g(\cdot)$  the Nemitsky (superposition) operator associated with $g$, that is,
$$N_g(u)(\cdot)=g(\cdot,u(\cdot))\ \mbox{for all}\ u\in W^{1,p}_{0}(\Omega).$$

Evidently, the mapping $z\mapsto N_g(u)(z)$ is measurable.

\section{Multiplicity Theorems}

In this section we prove two multiplicity theorems for problem \eqref{eqP} when $\lambda>0$ is big. In both theorems we provide precise sign information for all  solutions. Our method of proof is based on a cut-off technique first used by Costa and Wang \cite{6} in the context of semilinear Dirichlet problems driven by the Laplacian.

The hypotheses on the reaction term $f(z,x)$ are the following:

\smallskip
$H_1:$ $f:\Omega\times\RR\rightarrow\RR$ is a Carath\'eodory function such that $f(z,0)=0$ for almost all $z\in\Omega$ and
\begin{itemize}
	\item[(i)] there exists $r\in(p,p^*)$ such that $\lim\limits_{x\rightarrow 0}\frac{f(z,x)}{|x|^{r-2}x}=0$ uniformly for almost all $z\in\Omega$,
	\item[(ii)] if $F(z,x)=\int^x_0f(z,s)ds$, then there exists $\beta\in(r,p^*)$ such that $\lim\limits_{x\rightarrow 0}\frac{F(z,x)}{|x|^{\beta}}=+\infty$ uniformly for almost all $z\in\Omega$;
	\item[(iii)] there exist $q\in(p,p^*)$ and $\delta>0$ such that
	$$0<q F(z,x)\leq f(z,x)x\ \mbox{for almost all}\ z\in\Omega\ \mbox{and all}\ 0<|x|\leq\delta ;$$
	\item[(iv)] there exists $\hat{\xi}>0$ such that for almost all $z\in\Omega$, the function
	$$x\mapsto f(z,x)+\hat{\xi}|x|^{p-2}x$$
	is nondecreasing on $[-\delta,\delta]$ (here $\delta>0$ is as in (iii) above).
\end{itemize}

Hypotheses $H_1(i),(ii)$ imply that we can find $\delta_1\in\left(0,\delta\right]$ such that
\begin{equation}\label{eq2}
	|f(z,x)|\leq|x|^{r-1}\ \mbox{and}\ F(z,x)\geq |x|^{\beta}\ \mbox{for almost all}\ z\in\Omega\ \mbox{and all}\ |x|\leq\delta_1.
\end{equation}

Let $\eta\in\left(0,\frac{\delta_1}{2}\right)$ and let $\vartheta\in C^2(\RR,[0,1])$ be an even cut-off function such that
\begin{eqnarray}
	&&\vartheta(x)=\left\{\begin{array}{ll}
		1&\mbox{if}\ |x|\leq\eta\\
		0&\mbox{if}\ |x|>2\eta
	\end{array}\right. \label{eq3}\\
	&&x\vartheta'(x)\leq 0,\ |x\vartheta'(x)|\leq\frac{2}{\eta}\ \mbox{for all}\ x\in\RR.\label{eq4}
\end{eqnarray}

Using this cut-off function, we introduce the following modification of the primitive $F(z,x):$
$$\hat{F}(z,x)=\vartheta(x)F(z,x)+(1-\vartheta(x))\frac{|x|^r}{r}.$$

Also, we set
$$\hat{f}(z,x)=\hat{F}'_x(z,x)=\frac{\partial\hat{F}}{\partial x}(z,x).$$

By Lemma 1.1 of Costa and Wang \cite{6}, we have the following property.
\begin{lemma}\label{lem6}
	If hypotheses $H_1$ hold, then
	\begin{itemize}
		\item[(a)] $|\hat{f}(z,x)|\leq c_1|x|^{r-1}$ for almost all $z\in\Omega$, all $x\in\RR$ and some $c_1>0$;
		\item[(b)] $0<\mu\hat{F}(z,x)\leq\hat{f}(z,x)x$ for almost all $z\in\Omega$ and all $x\in\RR\backslash\{0\}$ with $\mu=\min\{q,r\}$.
	\end{itemize}
\end{lemma}

We introduce the following auxiliary Dirichlet problem
\begin{equation}
	-\Delta_pu(z)-\Delta u(z)=\lambda\hat{f}(z,u(z))\ \mbox{in}\ \Omega,\ u|_{\partial\Omega}=0. \tag{$Q_{\lambda}$}\label{eqQ}
\end{equation}

The nonlinear regularity theory (see Lieberman \cite{15}, Theorem \ref{th1} and Motreanu, Motreanu and Papageorgiou \cite[Corollary 8.7, p. 208]{16}) together with Lemma \ref{lem6}, give the following result (see also Papageorgiou and R\u adulescu \cite{19} for an alternative approach).
\begin{prop}\label{prop7}
	If hypotheses $H_1$ hold and $u_{\lambda}\in W^{1,p}_{0}(\Omega)\ (\lambda>0)$ is a solution of \eqref{eqQ}, then $u_{\lambda}\in C^1(\overline{\Omega})$ and there exists $c_2=c_2(r,N,\Omega)>0$ such that
	$$||u_{\lambda}||_{\infty}\leq c_2\lambda^{\frac{1}{p^*-r}}||u||^{\frac{p^*-p}{p^*-r}}.$$
\end{prop}

For every $\lambda>0$, we consider the energy functional $\hat{\varphi}_{\lambda}:W^{1,p}_{0}(\Omega)\rightarrow\RR$ for problem \eqref{eqQ} defined by
$$\hat{\varphi}_{\lambda}(u)=\frac{1}{p}||Du||^p_p+\frac{1}{2}||Du||^2_2-\lambda\int_{\Omega}\hat{F}(z,u)dz\ \mbox{for all}\ u\in W^{1,p}_{0}(\Omega).$$

Evidently, $\hat{\varphi}_{\lambda}\in C^1(W^{1,p}_{0}(\Omega),\RR)$. From Lemma \ref{lem6}(b) we see that $\hat{F}(z,\cdot)$ satisfies a global Ambrosetti-Rabinowitz condition (see Ambrosetti and Rabinowitz \cite{3}). From this we derive the following result.
\begin{prop}\label{prop8}
	If hypotheses $H_1$ hold and $\lambda>0$, then $\hat{\varphi}_{\lambda}$ satisfies the PS-condition.
\end{prop}

Next, we show that for big $\lambda>0$ the energy functional $\hat{\varphi}_{\lambda}$ satisfies the mountain pass geometry (see Theorem \ref{th1}).
\begin{prop}\label{prop9}
	If hypotheses $H_1$ hold, then
	\begin{itemize}
		\item[(a)] for every $\lambda>0$, we can find $\hat{m}_{\lambda}>0$ and $\hat{\rho}_{\lambda}>0$ such that
		$$\hat{\varphi}_{\lambda}(u)\geq\hat{m}_{\lambda}>0\ \mbox{for all}\ u\in W^{1,p}_{0}(\Omega)\ \mbox{with}\ ||u||=\hat{\rho}_{\lambda};$$
		\item[(b)] we can find $\hat{\lambda}_*>0$ and $\bar{u}\in C^1_0(\overline{\Omega})$ such that for all $\lambda\geq\hat{\lambda}_*$ we have
		\begin{eqnarray*}
			&&|\bar{u}(z)|\leq\frac{\eta}{2}\ \mbox{for all}\ z\in\overline{\Omega},\ ||\bar{u}||>\rho_{\lambda},\ \hat{\varphi}_{\lambda}(\bar{u})\leq 0<\hat{m}_{\lambda},\ ||\bar{u}||^p+c^2_3||\bar{u}||^2\leq\lambda||\bar{u}||^{\beta}_{\beta},
		\end{eqnarray*}
		with $c_3>0$ being such that $||\cdot||_{1,2}\leq c_3||\cdot||$ (here $||\cdot||_{1,2}$ denotes the norm of $H^1_0(\Omega)$, recall that $p>2$).
	\end{itemize}
\end{prop}
\begin{proof}
	$(a)$ By Lemma \ref{lem6}(a) we have
	$$\hat{F}(z,x)\leq c_4|x|^r\ \mbox{for almost all}\ z\in\Omega,\ \mbox{all}\ x\in\RR\ \mbox{and some}\ c_4>0.$$
	
	Then for all $u\in W^{1,p}_{0}(\Omega)$ we have
	\begin{eqnarray*}
	    \hat{\varphi}_{\lambda}(u)&\geq&\frac{1}{p}||Du||^p_p-\lambda c_4||u||^r_r\\
	    &\geq&\frac{1}{p}||u||^p-\lambda c_5||u||^r\ \mbox{for some}\ c_5>0\ (\mbox{recall that}\ r<p^*)\\
	    &=&\left[\frac{1}{p}-\lambda c_5||u||^{r-p}\right]||u||^p.
	\end{eqnarray*}

    If we choose $\hat{\rho}_{\lambda}=\left(\frac{1}{2p\lambda c_5}\right)^{\frac{1}{r-p}}>0$, then
    $$\hat{\varphi}_{\lambda}(u)\geq\frac{1}{2p}\hat{\rho}_{\lambda}^p>0\ \mbox{for all}\ u\in W^{1,p}_{0}(\Omega)\ \mbox{with}\ ||u||=\hat{\rho}_{\lambda}.$$

    $(b)$ Let $\bar{u}\in C^1_0(\overline{\Omega})\backslash\{0\}$ be such that
    \begin{equation}\label{eq5}
        |\bar{u}(z)|\leq\frac{\eta}{2}\ \mbox{for all}\ z\in\overline{\Omega}.
    \end{equation}

    Note that $\hat{\rho}_{\lambda}\rightarrow 0$ as $\lambda\rightarrow +\infty$. So, we can find $\hat{\lambda}_*>0$ such that
    \begin{equation}\label{eq6}
    ||\bar{u}||>\hat{\rho}_{\lambda}\ \mbox{and}\ ||\bar{u}||^p+c^2_3||\bar{u}||^2\leq\lambda||\bar{u}||^{\beta}_{\beta}\ \mbox{for all}\ \lambda>\hat{\lambda}_*\,.
    \end{equation}

Using (\ref{eq2}) and (\ref{eq5}), we have
\begin{eqnarray*}
    \hat{\varphi}_{\lambda}(\bar{u})&\leq&\frac{1}{p}||D\bar{u}||^p_p+\frac{1}{2}||D\bar{u}||^2_2-\lambda||\bar{u}||^{\beta}_{\beta}\\
    &\leq&||\bar{u}||^p+c_3^2||\bar{u}||^2-\lambda||\bar{u}||^{\beta}_{\beta}\leq 0<\hat{m}_{\lambda}\ \mbox{(see (\ref{eq5}))}.
\end{eqnarray*}
\end{proof}

\begin{prop}\label{prop10}
If hypotheses $H_1$ hold, $\lambda>0$ and $u\in K_{\hat{\varphi}_{\lambda}}$, then we can find $\hat{c}>0$ such that $||u||^p\leq\hat{c}\hat{\varphi}_{\lambda}(u)$.
\end{prop}
\begin{proof}
Let $u\in K_{\hat{\varphi}_{\lambda}}$. We have
\begin{eqnarray*}
    &&\mu\hat{\varphi}_{\lambda}(u)=\mu\hat{\varphi}_{\lambda}(u)-\left\langle \hat{\varphi}'_{\lambda}(u),u\right\rangle\ (\mbox{since}\ u\in K_{\hat{\varphi}_{\lambda}})\\
    &=&\frac{\mu}{p}||Du||^p_p+\frac{\mu}{2}||Du||^2_2-\int_{\Omega}\mu\lambda\hat{F}(z,u)dz-||Du||^p_p-||Du||^2_2+\int_{\Omega}\lambda\hat{f}(z,u)udz\\
    &=&\left(\frac{\mu}{p}-1\right)||Du||^p_p+\left(\frac{\mu}{2}-1\right)||Du||^2_2+\lambda\int_{\Omega}[\hat{f}(z,u)u-\mu\hat{F}(z,u)]dz\\
    &\geq&c_6||u||^p\ \mbox{with}\ c_6=\frac{\mu}{p}-1>0\ (\mbox{recall that}\ \mu>p>2\ \mbox{and use Lemma \ref{lem6}(b)})\\
    \Rightarrow&&||u||^p\leq\hat{c}\hat{\varphi}_{\lambda}(u)\ \mbox{with}\ \hat{c}=\frac{\mu}{c_6}>0.
\end{eqnarray*}
\end{proof}

Now we can produce a nontrivial smooth solution for the auxiliary problem \eqref{eqQ} when $\lambda>0$ is big.
\begin{prop}\label{prop11}
If hypotheses $H_1$ hold, then for every $\lambda\geq\hat{\lambda}_*$ (see Proposition \ref{prop9}(b)) problem \eqref{eqQ} admits a nontrivial smooth solution $u_{\lambda}\in C^1_0(\overline{\Omega})$ such that
$$||u_{\lambda}||_{\infty}\leq\tilde{c}\frac{1}{\lambda^{\frac{2}{\beta-2}}}\ \mbox{for some}\ \tilde{c}>0.$$
\end{prop}

\begin{proof}
Propositions \ref{prop8} and \ref{prop9} permit the use of Theorem \ref{th1} (the mountain pass theorem). So, we can find $u_{\lambda}\in W^{1,p}_0(\Omega)$ such that
\begin{equation}\label{eq7}
u_{\lambda}\in K_{\hat{\varphi}_{\lambda}}\ \mbox{and}\ \hat{m}_{\lambda}\leq\hat{\varphi}_{\lambda}(u_{\lambda})=c_{\lambda}=\inf\limits_{\gamma\in\Gamma}\max\limits_{0\leq t\leq 1}\varphi_{\lambda}(\gamma(t)),
\end{equation}
where $\Gamma=\{\gamma\in C([0,1],W^{1,p}_0(\Omega)):\gamma(0)=0,\gamma(1)=\bar{u}\}$.

We consider the function
$$\tilde{\xi}_{\lambda}(t)=t^2[||\bar{u}||^p+c_3^2||\bar{u}||^2]-\lambda t^{\beta}||\bar{u}||^{\beta}_{\beta}\ \mbox{for all}\ t\geq 0.$$

Note that $\tilde{\xi}_{\lambda}(\cdot)$ is continuous and
$$\tilde{\xi}_{\lambda}(0)=0,\ \tilde{\xi}_{\lambda}(1)\leq 0\ (\mbox{see (\ref{eq6})}).$$

Since $2<p<\beta$, for small $t\in(0,1)$ we see that
$$\tilde{\xi}_{\lambda}(t)>0.$$

Therefore we can find $t_0\in(0,1)$ such that
\begin{eqnarray*}
&&\tilde{\xi}_{\lambda}(t_0)=\max\limits_{0\leq t\leq 1}\xi_{\lambda}(t),\\
&\Rightarrow &\tilde{\xi}_{\lambda}'(t_0)=2t_0[||\bar{u}||^p+c_3^2||\bar{u}||^2]-\lambda\beta t_0^{\beta-1}||\bar{u}||^{\beta}_{\beta}=0,\\
&\Rightarrow &t_0=\left[\frac{2(||\bar{u}||^p+c^2_3||\bar{u}||^2)}{\lambda\beta||\bar{u}||^{\beta}_{\beta}}\right]^{\frac{1}{\beta-2}}.
\end{eqnarray*}

Then we have
\begin{eqnarray}\label{eq8}
\tilde{\xi}_{\lambda}(t_0)&=&\frac{1}{\lambda^{\frac{2}{\beta-2}}}c_7+\lambda\frac{1}{\lambda^{\frac{\beta}{\beta-2}}}c_8\ \mbox{for some}\ c_7,c_8>0\nonumber\\
&=&\frac{1}{\lambda^{\frac{2}{\beta-2}}}c_9\ \mbox{with}\ c_9=c_7+c_8>0.
\end{eqnarray}

Let $\gamma_0(t)=t\bar{u}$. Then $\gamma_0\in\Gamma$ and so by virtue of (\ref{eq7}) we have
\begin{eqnarray*}
&\hat{\varphi}_{\lambda}(u_{\lambda})&=c_{\lambda}\leq\max\limits_{0\leq t\leq 1}\hat{\varphi}_{\lambda}(t\bar{u})\leq\max\limits_{0\leq t\leq 1}\tilde{\xi}_{\lambda}(t)\ (\mbox{see (\ref{eq5}) and (\ref{eq2})})\\
&&=\tilde{\xi}_{\lambda}(t_0)=\frac{c_9}{\lambda^{\frac{2}{\beta-2}}}\ (\mbox{see (\ref{eq8})}),\\
\Rightarrow&||u_{\lambda}||^p&\leq\frac{\hat{c}c_9}{\lambda^{\frac{2}{\beta-2}}}\ \mbox{and so}\ \tilde{c}=\hat{c}c_9>0\ (\mbox{see Proposition \ref{prop10}}).
\end{eqnarray*}
\end{proof}

Now we are ready to produce constant sign smooth solutions for problem \eqref{eqP} when the parameter $\lambda>0$ is big.
\begin{prop}\label{prop12}
If hypotheses $H_1$ hold, then we can find $\lambda_*\geq\hat{\lambda}_*>0$ such that for all $\lambda\geq\lambda_*$, problem \eqref{eqP} has at least two nontrivial smooth solutions of constant sign
$$\hat{u}_{\lambda}\in D_+\ \mbox{and}\ \hat{v}_{\lambda}\in-D_+.$$
\end{prop}
\begin{proof}
First we produce the positive solution.

To this end we consider the $C^1$-functional $\hat{\varphi}^+_{\lambda}:W^{1,p}_{0}(\Omega)\rightarrow\RR $ defined by
$$\hat{\varphi}^+_{\lambda}(u)=\frac{1}{p}||Du||^p_p+\frac{1}{2}||Du||^2_2-\lambda\int_{\Omega}\hat{F}(z,u^+)dz\ \mbox{for all}\  u\in W^{1,p}_{0}(\Omega).$$

Note that $\hat{F}(z,x^+)=0$ for almost all $z\in\Omega$ and all $x\leq 0$. By Lemma \ref{lem6}, $\hat{F}_+(z,x)=\hat{F}(z,x^+)$ satisfies the Ambrosetti-Rabinowitz condition on $\RR_+=\left[0,+\infty\right)$. Therefore
\begin{equation}\label{eq9}
\hat{\varphi}^+_{\lambda}\ \mbox{satisfies the PS-condition}.
\end{equation}

A careful reading of the proof of Proposition \ref{prop9} reveals that the result remains true also for the functional $\hat{\varphi}^+_{\lambda}$ (in this case in part (b) we choose $\bar{u}\in C_+\setminus\{0\}$). This fact and (\ref{eq9}) permit the use of Theorem \ref{th1} (the mountain pass theorem) and so we can find $\hat{u}_{\lambda}\in W^{1,p}_{0}(\Omega)$ such that
 \begin{equation}\label{eq10}
 \hat{u}_{\lambda}\in K_{\hat{\varphi}^+_{\lambda}}\ \mbox{and}\ \hat{\varphi}^+_{\lambda}(0)=0<\hat{m}^+_{\lambda}\leq\hat{\varphi}^+_{\lambda}(\hat{u}_{\lambda}).
 \end{equation}

 From (\ref{eq9}) we see that
 $$\hat{u}_{\lambda}\neq 0\ \mbox{and}\ (\hat{\varphi}^+_{\lambda})'(\hat{u}_{\lambda})=0.$$

 So, we have
 \begin{equation}\label{eq11}
 \left\langle A_p(\hat{u}_{\lambda}),h\right\rangle+\left\langle A(\hat{u}_{\lambda}),h\right\rangle=\int_{\Omega}\lambda\hat{f}(z,\hat{u}^+_{\lambda})hdz\ \mbox{for all}\ h\in W^{1,p}_{0}(\Omega).
 \end{equation}

In (\ref{eq11}) we choose $h=-\hat{u}^-_{\lambda}\in W^{1,p}_{0}(\Omega)$. Then
\begin{eqnarray*}
&&||D\hat{u}^-_{\lambda}||^p_p+||D\hat{u}^-_{\lambda}||^2_2=0,\\
&\Rightarrow&\hat{u}_{\lambda}\geq 0,\ \hat{u}_{\lambda}\neq 0.
\end{eqnarray*}

We have
\begin{equation}\label{eq12}
-\Delta_p\hat{u}_{\lambda}(z)-\Delta\hat{u}_{\lambda}(z)=\lambda\hat{f}(z,\hat{u}_{\lambda}(z))\ \mbox{for almost all}\ z\in\Omega,\left.\hat{u}_{\lambda}\right|_{\partial\Omega}=0.
\end{equation}

By Theorem 8.4 of Motreanu, Motreanu and Papageorgiou \cite[p. 204]{16}, we have
$$\hat{u}_{\lambda}\in L^{\infty}(\Omega).$$

So, we can apply Theorem 1 of Lieberman \cite{15} and infer that
$$\hat{u}_{\lambda}\in C_+\backslash\{0\}.$$

From Lemma \ref{lem6}(a) and (\ref{eq12}) we have
\begin{equation}\label{eq13}
\Delta_p\hat{u}_{\lambda}(z)+\Delta\hat{u}_{\lambda}(z)\leq c_1||\hat{u}_{\lambda}||^{r-p}_{\infty}\hat{u}_{\lambda}(z)^{p-1}\ \mbox{for almost all}\ z\in\Omega.
\end{equation}

  Let $a:\RR^N\rightarrow\RR^N$ be defined by
$$a(y)=|y|^{p-2}y+y\ \mbox{for all}\ y\in\RR^N.$$

Note that $a\in C^1(\RR^N,\RR^N)$ (recall that $p>2$) and
$${\rm div}\, a(Du) = \Delta_pu+\Delta u\ \mbox{for all}\ u\in W^{1,p}_0(\Omega).$$

We have
\begin{eqnarray*}
	& & \nabla a(y) = |y|^{p-2}\left[I+(p-2)\frac{y\otimes y}{|y|^2}\right] + I\ \mbox{for all}\ y\in\RR^N\backslash\{0\}, \\
	& \Rightarrow & (\nabla a(y)\xi^{'},\xi^{'})_{\RR^N}\ge|\xi^{'}|^2\ \mbox{for all}\ \xi^{'}\in\RR^N.
\end{eqnarray*}

This permits the use of the tangency principle of Pucci and Serrin \cite[p. 35]{22}  and implies that
$$\hat{u}_\lambda(z)>0\ \mbox{for all}\ z\in\Omega.$$

Then (\ref{eq13}) and the boundary point theorem of Pucci and Serrin \cite[p. 120]{22} imply that
$$\hat{u}_\lambda\in D_+.$$

Note that
\begin{eqnarray*}
	& & \hat{\varphi}_\lambda^{'}|_{C_+} = (\hat{\varphi}_\lambda^{+})^{'}|_{C_+}, \\
	& \Rightarrow & \hat{u}_\lambda\in K_{\hat{\varphi}_\lambda}, \\
	& \Rightarrow & ||\hat{u}_\lambda||_\infty \le \tilde{c} \lambda^{-\frac{2}{\beta-2}}\ \mbox{for all}\ \lambda\ge\hat{\lambda}_*\ \mbox{(see Proposition \ref{prop11})}.
\end{eqnarray*}

It follows that
$$||\hat{u}_\lambda||_\infty\rightarrow0\ \mbox{as}\ \lambda\rightarrow+\infty,\ \lambda\ge\hat{\lambda}_+\ \mbox{(recall that $\beta>2$)}.$$

So, we can find $\lambda_*\ge\hat{\lambda}_*$ such that
\begin{eqnarray*}
	& & 0\le\hat{u}_\lambda(z)\le\eta\ \mbox{for all}\ z\in\overline{\Omega},\ \mbox{all}\ \lambda\ge\lambda_*, \\
	& \Rightarrow & \hat{u}_\lambda\in D_+\ \mbox{is a positive solution of \eqref{eqP} for}\ \lambda\ge\lambda_*\ \mbox{(see (\ref{eq3}))}.
\end{eqnarray*}

Similarly we obtain a negative solution
$$\hat{v}_\lambda\in-D_+\ \mbox{for all}\ \lambda\ge\lambda_*$$
(we may need to increase $\lambda_*\ge\hat{\lambda}_*$). In this case we work with the $C^{1}$-functional $\hat{\varphi}_\lambda^-:W^{1,p}_0(\Omega)\rightarrow\RR$ defined by
$$\hat{\varphi}_\lambda^-(u)=\frac{1}{p}||Du||^p_p + \frac{1}{2}||Du||^2_2 - \lambda\int_\Omega\hat{F}(z,-u^-)dz\ \mbox{for all}\ u\in W^{1,p}_0(\Omega).$$
\end{proof}

In fact, we can produce extremal constant sign solutions for \eqref{eqP} when $\lambda\ge\lambda_*$, that is, a smallest positive solution and a biggest negative solution. These extremal constant sign solutions will be useful in producing nodal (that is, sign-changing) solutions.

We have the following result.
\begin{prop}\label{prop13}
	If hypotheses $H_1$ hold and $\lambda\ge\lambda_*$ (see Proposition \ref{prop12}), then problem \eqref{eqP} admits a smallest positive solution $u^*_\lambda\in D_+$ and a biggest negative solution $v^*_\lambda\in-D_+$.
\end{prop}
\begin{proof}
	We introduce the following two sets
	\begin{eqnarray*}
		& S_+ = \left\{u\in W^{1,p}_0(\Omega): u\in[0,\eta], u\ \mbox{is a positive solution of}\ \eqref{eqP} \right\}, \\
		& S_- = \left\{v\in W^{1,p}_0(\Omega): v\in[-\eta,0], v\ \mbox{is a negative solution of}\ \eqref{eqP} \right\}.
	\end{eqnarray*}
	
	From Proposition \ref{prop12} and its proof, we infer that
	$$\emptyset\neq S_+\subseteq D_+\ \mbox{and}\ \emptyset\neq S_-\subseteq-D_+.$$
	
	Invoking Lemma 3.10 of Hu and Papageorgiou \cite[p. 178]{13}, we can find $\{u_n\}_{n\ge1}\subseteq S_+$ such that
	$$\inf_{n\ge1}u_n=\inf S_+.$$
	
	We have
	\begin{equation}\label{eq14}
		\langle A_p(u_n),h\rangle + \langle A(u_n),h\rangle = \lambda\int_\Omega f(z,u_n)hdz\ \mbox{for all}\ h\in W^{1,p}_0(\Omega),\ \mbox{all}\ n\in\NN.
	\end{equation}
	
	Evidently, $\{u_n\}_{n\ge1}\subseteq W^{1,p}_0(\Omega)$ is bounded and so we may assume that
	\begin{equation}\label{eq15}
		u_n\xrightarrow{w}u_\lambda^*\ \mbox{in}\ W^{1,p}_0(\Omega)\ \mbox{and}\ u_n\rightarrow u_\lambda^*\ \mbox{in}\ L^p(\Omega).
	\end{equation}
	
	In (\ref{eq14}) we choose $h=u_n-u^*_\lambda\in W^{1,p}_0(\Omega)$, pass to the limit as $n\rightarrow\infty$ and use (\ref{eq15}). We obtain
	\begin{eqnarray}
		& & \lim_{n\to\infty}\left[ \langle A_p(u_n),u_n-u^*_\lambda\rangle + \langle A(u_n), u_n-u^*_\lambda\rangle \right] = 0 \nonumber \\
		& \Rightarrow & \limsup_{n\to\infty}\left[ \langle A_p(u_n), u_n-u^*_\lambda\rangle + \langle A(u^*_\lambda),u_n-u^*_\lambda\rangle \right] \le 0 \nonumber \\
		& & \mbox{(exploiting the monotonicity of $A(\cdot)$)} \nonumber \\
		& \Rightarrow & \limsup_{n\to\infty} \langle A_p(u_n), u_n-u^*_\lambda\rangle \le 0\ \mbox{(see (\ref{eq15}))} \nonumber \\
		& \Rightarrow & u_n\rightarrow u^*_\lambda\ in\ W^{1,p}_0(\Omega)\ \mbox{(see Proposition \ref{prop5} and (\ref{eq15}))},\ u^*_\lambda\ge0. \label{eq16}
	\end{eqnarray}
	
	If in (\ref{eq14}) we pass to the limit as $n\to\infty$ and use (\ref{eq16}) then
	\begin{eqnarray}
		& & \langle A_p(u^*_\lambda),h\rangle + \langle A(u^*_\lambda),h\rangle = \lambda\int_\Omega f(z,u^*_\lambda)hdz\ \mbox{for all}\ h\in W^{1,p}_0(\Omega), \nonumber \\
		& \Rightarrow & -\Delta_p u^*_\lambda(z) - \Delta u^*_\lambda(z) = \lambda f(z,u^*_\lambda(z))\ \mbox{for almost all}\ z\in\Omega,\ u^*_\lambda|_{\partial\Omega}=0,\ u^*_\lambda\ge0. \label{eq17}
	\end{eqnarray}
	
	By (\ref{eq17}) we see that if we can show that $u^*_\lambda\neq0$, then we have $u^*_\lambda\in S_+$. We argue by contradiction. So, suppose that $u^*_\lambda=0$. Then by (\ref{eq16}) we have
	\begin{equation}\label{eq18}
		u_n\rightarrow0\ \mbox{in}\ W^{1,p}_0(\Omega).
	\end{equation}
	
	Let $y_n=\frac{u_n}{||u_n||},\ n\in\NN$. Then $||y_n||=1,\ y_n\ge0$ for all $n\in\NN$ and so we may assume that
	\begin{equation}\label{eq19}
		y_n\xrightarrow{w}y\ \mbox{in}\ W^{1,p}_0(\Omega)\ \mbox{and}\ y_n\rightarrow y\ \mbox{in}\ L^p(\Omega).
	\end{equation}
	
	By (\ref{eq14}) we have
	\begin{equation}\label{eq20}
		||u_n||^{p-2}\langle A_p(y_n),h\rangle + \langle A(y_n),h\rangle = \lambda\int_\Omega\frac{N_f(u_n)}{||u_n||}hdz\ \mbox{for all}\ h\in W^{1,p}_0(\Omega).
	\end{equation}
	
	Hypothesis $H_1(i)$ implies that
	\begin{equation}\label{eq21}
		\int_\Omega\frac{N_f(u_n)}{||u_n||}hdz\rightarrow0\ \mbox{for all}\ h\in W^{1,p}_0(\Omega).
	\end{equation}
	
	Since $A_p(\cdot)$ is bounded (that is, maps bounded sets to bounded sets), it follows from (\ref{eq18}) and (\ref{eq19})  that
	\begin{equation}\label{eq22}
		||u_n||^{p-2} \langle A_p(y_n),h\rangle\rightarrow0\ \mbox{for all}\ h\in W^{1,p}_0(\Omega).
	\end{equation}
	
	In (\ref{eq20}) we choose $h=y_n-y\in W^{1,p}_0(\Omega)$, pass to the limit as $n\to\infty$ and use (\ref{eq21}), (\ref{eq22}). Then
	\begin{equation}\label{eq23}
		\lim_{n\to\infty}\langle A(y_n), y_n-y\rangle = 0.
	\end{equation}
	
	Note that
	\begin{eqnarray}
		& & \langle A(y_n),y_n-y\rangle - \langle A(y), y_n-y\rangle = ||D(y_n-y)||^2_2\ \mbox{for all}\ n\in\NN, \nonumber \\
		& \Rightarrow & ||D(y_n-y)||^2_2\to0\ \mbox{(from (\ref{eq19}), (\ref{eq23}))}, \label{eq24} \\
		& \Rightarrow & ||y||_{H^1_0(\Omega)} = 1. \label{eq25}
	\end{eqnarray}
	
	On the other hand from (\ref{eq20}), passing to the limit as $n\to\infty$ and using (\ref{eq21}), (\ref{eq22}), (\ref{eq24}), we obtain
	\begin{eqnarray*}
		& & \langle A(y),h\rangle = 0\ \mbox{for all}\ h\in W^{1,p}_0(\Omega), \\
		& \Rightarrow & y=0,\ \mbox{which contradicts (\ref{eq25})}.
	\end{eqnarray*}

Therefore $u^*_\lambda\neq0$ and so
	$$u^*_\lambda\in S_+\ \mbox{and}\ u^*_\lambda=\inf S_+.$$
	
	Similarly, we produce $v^*_\lambda\in W^{1,p}_0(\Omega)$ such that
	$$v^*_\lambda\in S_-\ \mbox{and}\ v^*_\lambda=\sup S_-.$$
The proof is now complete.
\end{proof}

Now our strategy becomes clear. We will truncate the reaction term at $\{v^*_\lambda(z),u^*_\lambda(z)\}$ in order to focus on the order interval $$[v^*_\lambda,u^*_\lambda]=\{u\in W^{1,p}(\Omega):v^*_\lambda(z)\leq u(z)\leq u^*_\lambda(z)\ \mbox{for almost all}\ z\in \Omega\}.$$
Working with the truncated functional and using variational tools (critical point theory), we will produce a nontrivial solution $y_\lambda\in [v^*_\lambda,u^*_\lambda]\cap C^1_0(\overline{\Omega}), y_\lambda\notin \{v^*_\lambda,u^*_\lambda\}$. The extremality of $v^*_\lambda$ and $u^*_\lambda$, forces $y_\lambda$ to be nodal.
\begin{prop}\label{prop14}
	If hypotheses $H_1$ hold, then there exists $\lambda^0_*\geq\lambda_*$ such that for all $\lambda>\lambda^0_*$ problem \eqref{eqP} admits a nodal solution $y_\lambda\in {\rm int}_{C^1_0(\overline{\Omega})}[v^*_\lambda,u^*_\lambda]$ (that is, $y_\lambda\in C^1_0(\overline{\Omega})$ and $u^*_\lambda-y_\lambda,y_\lambda-v^*_\lambda\in D_+$).
\end{prop}
\begin{proof}
	Let $u^*_\lambda\in D_+$ and $v^*_\lambda\in-D_+$ be the two extremal constant sign solutions produced by Proposition \ref{prop13}. We introduce the Carath\'eodory function $\hat{g}(z,x)$ defined by
	\begin{equation}\label{eq26}
		\hat{g}(z,x)=\left\{\begin{array}{ll}
			f(z,v^*_\lambda(z)) & \mbox{if}\ x<v^*_\lambda(z)\\
			f(z,x)				& \mbox{if}\ v^*_\lambda(z)\leq x \leq u^*_\lambda(z)\\
			f(z,u^*_\lambda(z))	& \mbox{if}\ u^*_\lambda(z)<x.
		\end{array}\right.
	\end{equation}
	
	We also consider the positive and negative truncations of $\hat{g}(z,\cdot)$, namely the Carath\'eodory functions
	$$\hat{g}_\pm(z,x)=\hat{g}(z,\pm x^\pm).$$
	
	We set $\hat{G}(z,x)=\int^x_0\hat{g}(z,s)ds$ and $\hat{G}_\pm(z,x)=\int^x_0\hat{g}_\pm(z,s)ds$ and consider the $C^1$-functionals $\hat{\sigma}_\lambda, \hat{\sigma}^\pm_\lambda: W^{1,p}_0(\Omega)\rightarrow\RR$ defined by
	\begin{eqnarray*}
		\hat{\sigma}_\lambda(u) & = & \frac{1}{p}||Du||^p_p+\frac{1}{2}||Du||^2_2-\lambda\int_\Omega\hat{G}(z,u)dz,\\
		\hat{\sigma}^\pm_\lambda(u) & = & \frac{1}{p}||Du||^p_p+\frac{1}{2}||Du||^2_2-\lambda\int_\Omega\hat{G}_\pm(z,u)dz\ \mbox{for all}\ u\in W^{1,p}_0(\Omega).
	\end{eqnarray*}
	\begin{claim}\label{claim1}
		$K_{\hat{\sigma}_\lambda}\subseteq[v^*_\lambda, u^*_\lambda]\cap C^1_0 (\overline{\Omega}), K_{\hat{\sigma}^+_\lambda}=\{0,u^*_\lambda\}, K_{\hat{\sigma}^-_\lambda}=\{v^*_\lambda,0\}$.
	\end{claim}
	
	Let $u\in K_{\hat{\sigma}_\lambda}.$ Then
	\begin{equation}\label{eq27}
		\langle A_p(u),h\rangle + \langle A(u),h\rangle = \int_\Omega \lambda\hat{g}(z,u)hdz\ \mbox{for all}\ h\in W^{1,p}_0(\Omega).
	\end{equation}
	
	In (\ref{eq27}) we choose $h=(u-u^*_\lambda)^+\in W^{1,p}_0(\Omega)$. Then
	\begin{eqnarray*}
		&&\langle A_p(u), (u-u^*_\lambda)^+ \rangle + \langle A(u),(u-u^*_\lambda)^+ \rangle \\
		&=&\lambda\int_\lambda f(z,u^*_\lambda)(u-u^*_\lambda)^+dz\ \mbox{(see (\ref{eq26}))} \\
		&=&\langle A_p (u^*_\lambda), (u-u^*_\lambda)^+\rangle+\langle A(u),(u-u^*_\lambda)^+\rangle\ \mbox{(since } u^*_\lambda\in S_+),\\
		&\Rightarrow & \langle A_p(u)-A_p(u^*_\lambda),(u-u^*_\lambda)^+\rangle + \langle A(u)-A(u^*_\lambda), (u-u^*_\lambda)^+\rangle=0,\\
		&\Rightarrow &||D(u-u^*_\lambda)^+||^2_2=0,\\
		&\Rightarrow &u\leq u^*_\lambda.
	\end{eqnarray*}
	
	Similarly, if in (\ref{eq27}) we choose $h=(v^*_\lambda-u)^+\in W^{1,p}_0(\Omega)$, then we obtain
	$$v^*_\lambda\leq u.$$
	
	So, we have proved that
	$$u\in[v^*_\lambda, u^*_\lambda].$$
	
	Moreover, as before, the nonlinear regularity theory (see the proof of Proposition \ref{prop12}), implies that $u\in C^1_0(\overline{\Omega})$. Therefore we conclude that
	$$K_{\overline{\sigma}_\lambda}\subseteq[v^*_\lambda,u^*_\lambda]\cap C^1_0(\overline{\Omega}).$$
	
	In a similar fashion we show that
	$$K_{\hat{\sigma}^+_\lambda}\subseteq[0,u^*_\lambda]\cap C_+\ \mbox{and}\ K_{\hat{\sigma}^-_\lambda}\subseteq[v^*_\lambda,0]\cap(-C_+).$$
	
	The extremality of $u^*_\lambda$ and $v^*_\lambda$, implies that
	$$K_{\hat{\sigma}^+_\lambda}=\{0,u^*_\lambda\}\ \mbox{and}\ K_{\hat{\sigma}^-_\lambda}=\{v^*_\lambda,0\}.$$
	
	This proves Claim \ref{claim1}.
	
	On account of Claim \ref{claim1}, we see that we may assume that
	\begin{equation}\label{eq28}
		K_{\hat{\sigma}_\lambda}\ \mbox{is finite}.
	\end{equation}
	
	Otherwise we evidently already have an infinity of smooth nodal solutions and so we are done.
	\begin{claim}\label{claim2}
		$u^*_\lambda\in D_+$ and $v^*_\lambda\in -D_+$ are local minimizers of the functional $\hat{\sigma}_\lambda$.
	\end{claim}
	
	From (\ref{eq26}) it is clear that $\hat{\sigma}^+_\lambda$ is coercive. Also, $\hat{\sigma}^+_\lambda$ is sequentially weakly lower semicontinuous. So, we can find $\hat{u}^*_\lambda\in W^{1,p}_0(\Omega)$ such that
	\begin{equation}\label{eq29}
		\hat{\sigma}^+_\lambda(\hat{u}^*_\lambda)=\inf\left[\hat{\sigma}^+_\lambda(u):u\in W^{1,p}_0(\Omega)\right].
	\end{equation}
	
	Let $\hat{u}_1(p)$ be the positive principal eigenfunction of $(-\Delta_p,W^{1,p}_0(\Omega))$. We know that $\hat{u}_1(p)\in D_+$ (see Motreanu, Motreanu and Papageorgiou \cite{16}). Recall that $u^*_\lambda\in D_+$. So, by invoking Lemma 3.6 of Filippakis and Papageorgiou \cite{9}, we can find $\tau>0$ such that
	\begin{equation}\label{eq30}
		\tau\hat{u}_1(p)=\left[\frac{1}{2}u^*_\lambda,u^*_\lambda\right].
	\end{equation}
	
	Evidently, $\frac{1}{2}\leq\tau\leq2$. Hypothesis $H_1(ii)$ implies that there exists $\xi>0$ such that
	\begin{equation}\label{eq31}
		F(z,x)\geq\xi|x|^\beta\ \mbox{for almost all}\ z\in\Omega,\ \mbox{all}\ 0\leq x\leq \eta
	\end{equation}
	(here $\eta\in(0,\delta_1)$ is as in (\ref{eq3})). We have
	$$\hat{\sigma}^+_\lambda(\tau\hat{u}_1(p))\leq\frac{\tau^p}{p}\hat{\lambda}_1(p)+\frac{\tau^2}{2}||D\hat{u}_1(p)||^2_2-\lambda\xi\tau^\beta||\hat{u}_1(p)||^\beta_\beta,$$
	with $\hat{\lambda}_1(p)>0$ being the principal eigenvalue of $(-\Delta_p,W^{1,p}_0(\Omega)).$
	
	It follows that
	\begin{equation}\label{eq32}
		\begin{array}{c}
			\hat{\sigma}^+_\lambda(\tau\hat{u}_1(p))<0\\\\
			\mbox{if and only if}\\\\
			\frac{\frac{\tau^{p-2}}{p}\hat{\lambda}_1(p)+
\frac{1}{2}||D\hat{u}_1(p)||^2_2}{\xi\tau^{\beta-2}||\hat{u}_1(p)||^\beta_\beta}<\lambda.
		\end{array}
	\end{equation}
	
	Note that
	$$\frac{\frac{\tau^{p-2}}{p}\hat{\lambda}_1(p)+\frac{1}{2}||D\hat{u}_1(p)||^2_2}{\xi\tau^{\beta-2}||\hat{u}_1(p)||^\beta_\beta} \leq \frac{\frac{2^{\beta+p-4}}{p}\hat{\lambda}_1(p)+2^{\beta-3}||D\hat{u}_1(p)||^2_2}{\xi||\hat{u}_1(p)||^\beta_\beta}$$
	(recall that $\frac{1}{2}\leq\tau\leq2$).
	
	So, if we let $\lambda^0=\frac{\frac{2^{\beta+p-4}}{p}\hat{\lambda}_1(p)+2^{\beta-3}||D\hat{u}_1(p)||^2_2}{\xi||\hat{u}_1(p)||^\beta_\beta}$ and define
	$$\lambda^0_*=\max\{\lambda^0, \lambda_*\},$$
	then we infer from (\ref{eq32}) that
	\begin{eqnarray*}
		&&\hat{\sigma}^+_\lambda(\tau\hat{u}_1(p))<0\ \mbox{for all}\ \lambda>\lambda^0_*,\\
		&\Rightarrow & \hat{\sigma}^+_\lambda(\hat{u}^*_\lambda)<0=\hat{\sigma}^+_\lambda(0)\ \mbox{for all}\ \lambda>\lambda^0_*\ \mbox{(see (\ref{eq29}))},\\
		&\Rightarrow &\hat{u}^*_\lambda\neq0\ \mbox{and}\ \hat{u}^*_\lambda\in K_{\hat{\sigma}^+_\lambda}\ \mbox{for all}\ \lambda>\lambda^0_*\ \mbox{(see (\ref{eq29}))},\\
		&\Rightarrow &\hat{u}^*_\lambda=u^*_\lambda\ \mbox{for all}\ \lambda>\lambda^0_*\ \mbox{(see Claim \ref{claim1}).}
	\end{eqnarray*}
	
	By (\ref{eq26}) it is clear that $\hat{\sigma}^+_\lambda|_{C_+}=\hat{\sigma}_\lambda|_{C_+}$. Since $u^*_\lambda\in D_+$, it follows from  (\ref{eq29}) that
	\begin{eqnarray*}
		&&u^*_\lambda\ \mbox{is a local}\ C^1(\overline{\Omega})-\mbox{minimizer of}\ \hat{\sigma}_\lambda,\\
		&\Rightarrow & u^*_\lambda\ \mbox{is a local}\ W^{1,p}_0(\Omega)-\mbox{minimizer of}\ \hat{\sigma}_\lambda\ \mbox{(see Proposition \ref{prop2}).}
	\end{eqnarray*}
	
	Similarly for $v^*_\lambda\in -D_+$, using this time the functional $\hat{\sigma}^-_\lambda$.
	
	This proves Claim \ref{claim2}.
	
\smallskip
	Without any loss of generality, we may assume that
	$$\hat{\sigma}_\lambda(v^*_\lambda)\leq\hat{\sigma}(u^*_\lambda).$$
	
	By (\ref{eq28}) and Claim \ref{claim2}, we see that we can find small $\rho\in(0,1)$ such that
	\begin{equation}\label{eq33}
		\hat{\sigma}_\lambda(v^*_\lambda)\leq\hat{\sigma}_\lambda(u^*_\lambda)<\inf\left[ \hat{\sigma}_\lambda(u):||u-u^*_\lambda||=\rho \right]=\hat{m}_\lambda,\ ||v^*_\lambda-u^*_\lambda||>\rho
	\end{equation}
	(see Aizicovici, Papageorgiou and Staicu \cite{1}, proof of Proposition 29). The functional $\hat{\sigma}_\lambda$ is coercive (see (\ref{eq26})). Hence
	\begin{equation}\label{eq34}
		\hat{\sigma}_\lambda\ \mbox{satisfies the PS-condition}.
	\end{equation}
	
	Then (\ref{eq33}), (\ref{eq34}) permit the use of Theorem \ref{th1} (the mountain pass theorem). So, we can find $y_\lambda\in W^{1,p}_0(\Omega)$ such that
	\begin{equation}\label{eq35}
		y_\lambda\in K_{\hat{\sigma}_\lambda}\ \mbox{and}\ \hat{m}_\lambda\leq\hat{\sigma}_\lambda(y_\lambda).
	\end{equation}
	
	Using (\ref{eq33}), (\ref{eq35}) and Claim \ref{claim1}, we have
	\begin{eqnarray*}
		&&y_\lambda\in \left[ v^*_\lambda,u^*_\lambda \right]\cap C^1_0(\overline{\Omega}), \ y_\lambda\notin\{v^*_\lambda,u^*_\lambda\},\\
		&\Rightarrow & y_\lambda\ \mbox{is a smooth solution of}\ \eqref{eqP}.
	\end{eqnarray*}
	
	We need to show that $y_\lambda\neq0$, in order to conclude that $y_\lambda$ is a smooth nodal solution of \eqref{eqP} for $\lambda>\lambda^0_*$.
	
	From the previous argument we have that $y_\lambda$ is a critical point of mountain pass type for the functional $\hat{\sigma}_\lambda$. Therefore we have
	\begin{equation}\label{eq36}
		C_1(\hat{\sigma}_\lambda,y_\lambda)\neq0
	\end{equation}
	(see Motreanu, Motreanu and Papageorgiou \cite[Corollary 6.81, p. 168]{16}).
	
	Hypothesis $H_1(i)$ implies that we can find $\delta_2\in(0,\eta)$ such that
	\begin{equation}\label{eq37}
		F(z,x)\leq\frac{1}{r}|x|^r\ \mbox{for almost all}\ z\in\Omega,\ \mbox{all}\ |x|\leq\delta_2.
	\end{equation}
	
	Since $u^*_\lambda\in D_+$ and $v^*_\lambda\in -D_+$, we see that
	$$\mbox{int}_{C^1_0(\overline{\Omega})}[v^*_\lambda,u^*_\lambda]\neq\emptyset.$$
	
	So, we can find $\delta_3>0$ such that
	\begin{equation}\label{eq38}
		B^{C^1_0(\overline{\Omega})}_{\delta_3}=\{u\in C^1_0(\overline{\Omega}):||u||_{C^1_0(\overline{\Omega})}<\delta_3 \}\subseteq [v^*_\lambda,u^*_\lambda]\cap C^1_0(\overline{\Omega}).
	\end{equation}
	
	Then for $u\in B^{C^1_0(\overline{\Omega})}_{\delta_3}$, we have
	\begin{eqnarray*}
		\hat{\sigma}_\lambda(u) & \geq & \frac{1}{p}||Du||^p_p+\frac{1}{2}||Du||^2_2-\frac{1}{r}||u||^r_r\ \mbox{(see (\ref{eq37}), (\ref{eq38}))}\\
		& \geq & \frac{1}{p}||u||^p-c_{10}||u||^r\ \mbox{for some}\ c_{10}>0\ \mbox{(recall that $r<p^*$)}.
	\end{eqnarray*}
	
	Since $r>p$, we see that by choosing $\delta_3>0$ even smaller if necessary, we have
	\begin{eqnarray}
		& & \hat{\sigma}_\lambda(u)\geq0,\nonumber\\
		&\Rightarrow &u=0\ \mbox{is a local}\ C^1_0(\overline{\Omega})-\mbox{minimizer of}\ \hat{\sigma}_\lambda,\nonumber\\
		&\Rightarrow &u=0\ \mbox{is a local}\ W^{1,p}_0(\Omega)-\mbox{minimizer of}\ \hat{\sigma}_\lambda\ \mbox{(see Proposition \ref{prop2}),}\nonumber\\
		&\Rightarrow &C_k(\hat{\sigma}_\lambda,0)=\delta_{k,0}\ZZ\ \mbox{for all}\ k\in\NN_0.\label{eq39}
	\end{eqnarray}
	
	Comparing (\ref{eq36}) and (\ref{eq39}), we infer that
	\begin{eqnarray*}
		&&y_\lambda\neq0,\\
		&\Rightarrow &y_\lambda\in C^1_0(\overline{\Omega})\ \mbox{is a nodal solution of}\ \eqref{eqP}\ \mbox{for}\ \lambda>\lambda^0_*.
	\end{eqnarray*}
	
	Let $\hat{\xi}>0$ be as postulated by hypothesis $H_1(iv)$. Then for $x>y,\ x,y\in[-\eta,\eta]$ we have
	\begin{eqnarray*}
		f(z,x)-f(z,y) 	& \geq & -\hat{\xi}(|x|^{p-2}x-|y|^{p-2}y)\\
						& \geq & -\hat{\xi}c_{11}|x-y|\ \mbox{for some}\ c_{11}>0\ \mbox{(recall that $p>2$)}.
	\end{eqnarray*}
	
	Because of this inequality and since $u^*_\lambda,v^*_\lambda$ are solutions of \eqref{eqP}, $u^*_\lambda\neq v^*_\lambda$, we have from the tangency principle of Pucci and Serrin \cite[p. 35]{22},
	\begin{equation}\label{eq40}
		y_\lambda(z)<u^*_\lambda(z)\ \mbox{for all}\ z\in\Omega.
	\end{equation}
	
	Let $\hat{\xi}_0>\hat{\xi}$. We have
	\begin{eqnarray}\label{eq41}
		&   & -\Delta_p y_\lambda(z)-\Delta y_\lambda(z)+\lambda \hat{\xi}_0|y_\lambda(z)|^{p-2}y_\lambda(z) \nonumber\\
		& = & \lambda\left[ f(z,y_\lambda(z))+\hat{\xi}_0|y_\lambda(z)|^{p-2}y_\lambda(z) \right] \nonumber\\
		& = & \lambda\left[ f(z,y_\lambda(z))+\hat{\xi}|y_\lambda(z)|^{p-2}y_\lambda(z)+(\hat{\xi}_0-\hat{\xi})|y_\lambda(z)|^{p-2}y_\lambda(z) \right] \nonumber\\
		& \leq & \lambda \left[ f(z,u^*_\lambda(z))+\hat{\xi}u^*_\lambda(z)^{p-1}+(\hat{\xi}_0-\xi)u^*_\lambda(z)^{p-1} \right]\ \mbox{(see hypothesis $H_1(iv)$)}\nonumber \\
		& = & -\Delta_p u^*_\lambda(z)-\Delta u^*_\lambda(z)+\lambda\hat{\xi}_0u^*_\lambda(z)^{p-1}\ \mbox{for almost all}\ z\in\Omega\ \mbox{(since $u^*_\lambda\in S_+$)}.
	\end{eqnarray}
	
	Set $h_1(z)=f(z,y_\lambda(z))+\hat{\xi}|y_\lambda(z)|^{p-2}y_\lambda(z)+(\hat{\xi}_0-\hat{\xi})|y_\lambda(z)|^{p-2}y_\lambda(z)$ and\\
\indent		$\quad \ \ h_2(z)=f(z,u^*_\lambda(z))+\hat{\xi}u^*_\lambda(z)^{p-1}+(\hat{\xi}_0-\hat{\xi})u^*_\lambda(z)^{p-1}.$
	
	Since $u^*_\lambda,y_\lambda\in C^1(\overline{\Omega})$ and using hypothesis $H_1(iv)$ and (\ref{eq40}), we see that
	$$h_1\prec h_2.$$
	
	Then it follows from (\ref{eq41}) and Proposition \ref{prop3}  that
	$$u^*_\lambda-y_\lambda\in D_+.$$
	
	Similarly, we show that
	$$y_\lambda-v^*_\lambda\in D_+.$$
	
	Therefore
	$$y_\lambda\in {\rm int}_{C^1_0(\overline{\Omega})}[v^*_\lambda,u^*_\lambda].$$
This completes the proof.
\end{proof}

Se, we can state our first multiplicity theorem for problem \eqref{eqP}.
\begin{theorem}\label{th15}
	If hypotheses $H_1$ hold, then we can find $\lambda^0_*>0$ such that for all $\lambda > \lambda^0_*$ problem \eqref{eqP} has at least three nontrivial solutions
	$$\hat{u}_\lambda \in D_+,\ \hat{v}_\lambda \in -D_+\ and\ y_\lambda \in {\rm int}_{C_0^1(\overline{\Omega})}[\hat{v_\lambda},\hat{u}_\lambda]\ \mbox{nodal}.$$
\end{theorem}
	
We can improve this theorem and produce a second nodal solution provided we strengthen the conditions on $f(z,\cdot)$.

The new hypotheses on the reaction $f(z,x)$ are the following.

\smallskip
$H_2$: $f:\Omega\times\RR\rightarrow\RR$ is a measurable function such that for almost all $z\in\Omega$, $f(z, 0)=0$, $f(z,\cdot) \in C^1(\RR, \RR)$ and
\begin{itemize}
	\item [(i)] there exists $r\in(p,p^*)$ such that $\lim_{x\rightarrow0}\frac{f(z,x)}{|x|^{r-2}x}=0$ uniformly for almost all $z\in\Omega$;
	\item [(ii)] if $F(z,x)=\int_0^x f(z,s)ds$, then there exists $\beta\in(r,p^*)$ such that $\lim_{x\rightarrow0}\frac{F(z,x)}{|x|^\beta}=+\infty$ uniformly for almost all $z\in\Omega$;
	\item [(iii)] there exists $q\in(p,p^*)$ and $\delta>0$ such that
			\begin{eqnarray*}
				0<qF(z,x)\le f(z,x)x\ \mbox{for almost all}\ z\in\Omega,\ \mbox{all}\ 0<|x|\le\delta, \\
				|f'_x(z,x)|\le a_0(z)\ \mbox{for almost all}\ z\in\Omega\ \mbox{and all}\ |x|\le\delta\ \mbox{with}\ a_0\in L^\infty(\Omega).
			\end{eqnarray*}
\end{itemize}
\begin{remark}
	Evidently, hypothesis $H_1(i)$ implies that $f'_x(z,0)=0$ for almost all $z\in\Omega$. In the framework of the above conditions, hypothesis $H_1(iv)$ is automatically satisfied by the mean value theorem and hypothesis $H_2(iii)$. Therefore, hypotheses $H_2$ are a more restricted version of hypotheses $H_1$.
\end{remark}
\begin{theorem}\label{th16}
	If hypotheses $H_2$ hold, then there exists $\lambda_*^0>0$ such that for all $\lambda>\lambda_*^0$ problem \eqref{eqP} has at least four nontrivial smooth solutions
	\begin{eqnarray*}
		& & \hat{u}_\lambda\in D_+,\ \hat{v}_\lambda \in -D_+,\\
		& & y_\lambda, \hat{y}_\lambda \in {\rm int}_{C^1_0(\overline{\Omega})}[\hat{v}_\lambda,\hat{u}_\lambda]\ nodal.
	\end{eqnarray*}
\end{theorem}
\begin{proof}
	From Proposition \ref{prop10}, we know that we can find $\lambda^0_*>0$ such that for all $\lambda>\lambda^0_*$ problem \eqref{eqP} has at least three nontrivial smooth solutions
	\begin{equation}\label{eq42}
		\hat{u}_\lambda\in D_+,\ \hat{v}_\lambda\in -D_t\ \mbox{and}\ y_\lambda\in {\rm int}_{C_0^1(\overline{\Omega})}[\hat{v}_\lambda,\hat{u}_\lambda]\ \mbox{nodal}.
	\end{equation}
	
	We use the notation introduced in the proof of Proposition \ref{prop14}. By Claim \ref{claim2} of that proof, we know that $\hat{u}_\lambda$ and $\hat{v}_\lambda$ are both local minimizers of the functional $\hat{\sigma}_\lambda$. Therefore we have
	\begin{equation}\label{eq43}
		C_k(\hat{\sigma}_\lambda,\hat{u}_\lambda) = C_k(\hat{\sigma}_\lambda,\hat{v}_\lambda) = \delta_{k,0} \ZZ\ \mbox{for all}\ k\in\NN_0.
	\end{equation}
	
	Moreover, from (\ref{eq36}) we have
	\begin{equation}\label{eq44}
		C_1(\hat{\sigma}_\lambda,y_\lambda)\neq0.
	\end{equation}
	
	Consider the functional $\hat{\varphi}_\lambda\in C^2(W^{1,p}_0(\Omega),\RR)$ introduced before Proposition \ref{prop8} (note that  because of hypotheses $H_2$ and since $\rho\in C^2(\RR,[0,1])$, we have that $\hat{\varphi}_\lambda$ is $C^2$). We consider the homotopy
	$$h_\lambda(t,u) = (1-t)\hat{\sigma}_\lambda(u) + t\hat{\varphi}_\lambda(u)\ \mbox{for all}\ (t,u)\in[0,1]\times W^{1,p}_0(\Omega).$$
	
	Suppose we could find $\{t_n\}_{n\ge1}\subseteq[0,1]$ and $\{u_n\}_{n\ge1}\subseteq W^{1,p}_0(\Omega)$ such that
	\begin{equation}\label{eq45}
		t_n\rightarrow t\in [0,1], u_n\rightarrow y_\lambda\ \mbox{in}\ W^{1,p}_0(\Omega)\ \mbox{and}\ (h_\lambda)'_u(t_n,u_n)=0\ \mbox{for all}\ n\in\NN.
	\end{equation}
	
	From the equality in (\ref{eq45}), we have
	\begin{eqnarray*}
		& & \langle A_p(u_n),h\rangle + \langle A(u_n),h\rangle = (1-t_n)\lambda\int_\Omega\hat{g}(z,u_n)hdz + t_n\lambda\int_\Omega\hat{f}(z,u_n)hdz \\
		& & \mbox{for all}\ h\in W^{1,p}_0(\Omega),\ \mbox{all}\ n\in\NN
	\end{eqnarray*}
	\begin{eqnarray*}
		& & \Rightarrow -\Delta_p u_n(z) - \Delta u_n(z) = \lambda\left[(1-t_n)\hat{g}(z,u_n(z)) + t_n\hat{f}(z,u_n(z))\right]\\
		& &  \mbox{for almost all}\ z\in\Omega,\ u_n|_{\partial\Omega}=0\ \mbox{for all}\ n\in\NN.
	\end{eqnarray*}
	
	Corollary 8.6 of Motreanu, Motreanu and Papageorgiou \cite[p. 208]{16}, implies that we can find $c_{12}>0$ such that
	$$||u_n||_\infty\le c_{12}\ \mbox{for all}\ n\in\NN.$$
	
	Then Theorem 1 of Lieberman \cite{15} implies that we can find $\alpha\in(0,1)$ and $c_{13}>0$ such that
	\begin{equation}\label{eq46}
		u_n\in C_0^{1,\alpha}(\overline{\Omega}),\ ||u_n||_{C_0^{1,\alpha}(\overline{\Omega})}\le c_{13}\ \mbox{for all}\ n\in\NN.
	\end{equation}
	
	Exploiting the compact embedding of $C^{1,\alpha}_0(\overline{\Omega})$ into $C_0^1(\overline{\Omega})$, from (\ref{eq45}) and (\ref{eq46}) we infer that
	\begin{eqnarray*}
		& & u_n\rightarrow y_n\ \mbox{in}\ C^1_0(\overline{\Omega}),\\
		& \Rightarrow & u_n\in[\hat{v}_\lambda,\hat{u}_\lambda]\ \mbox{for all}\ n\ge n_0\ \mbox{(see (\ref{eq42}))},\\
		& \Rightarrow & \{u_n\}_{n\geq1}\subseteq K_{\hat{\sigma}_\lambda}\ \mbox{(see (\ref{eq26}) and recall the definition of $\hat{f}$)},
	\end{eqnarray*}
	a contradiction to our hypothesis that $K_{\hat{\sigma}_\lambda}$ is finite (see (\ref{eq28})). Hence (\ref{eq45}) cannot occur and from the homotopy invariance of critical groups (see Gasinski and Papageorgiou \cite[Theorem 5.125, p. 836]{12}), we have
	\begin{eqnarray}
		& & C_k(\hat{\sigma}_\lambda,y_\lambda) = C_k(\hat{\varphi}_\lambda,y_\lambda)\ \mbox{for all}\ k\in\NN_0, \label{eq47}\\
		& \Rightarrow & C_1(\hat{\varphi}_\lambda,y_\lambda)\neq0\ \mbox{(see (\ref{eq44}))}. \label{eq48}
	\end{eqnarray}
	
	Since $\hat{\varphi}_\lambda\in C^2(W^{1,p}_0(\Omega), \RR)$, from (\ref{eq48}) and Papageorgiou and R\u adulescu \cite{17} (see Proposition 3.5, Claim 3), we have
	\begin{eqnarray}
		& & C_k(\hat{\varphi}_\lambda,y_\lambda) = \delta_{k,1}\ZZ\ \mbox{for all}\ k\in\NN_0,\nonumber\\
		& \Rightarrow & C_k(\hat{\sigma}_\lambda,y_\lambda) = \delta_{k,1}\ZZ\ \mbox{for all}\ k\in\NN_0\
\mbox{(see \eqref{eq47})}.\label{eq49}
	\end{eqnarray}
	From the proof of Proposition 14, we know that $u=0$ is a local minimizer of $\hat{\sigma}_\lambda$. Hence
	\begin{equation}\label{eq50}
		C_k(\hat{\sigma}_\lambda,0) = \delta_{k,0}\ZZ\ \mbox{for all}\ k\in\NN_0.
	\end{equation}
	
	By (\ref{eq26}) it is clear that $\hat{\sigma}_\lambda$ is coercive. Hence
	\begin{equation}\label{eq51}
		C_k(\hat{\sigma}_\lambda,\infty) = \delta_{k,0}\ZZ\ \mbox{for all}\ k\in\NN_0.
	\end{equation}
	
	Suppose that $K_{\hat{\sigma}_\lambda}=\{\hat{u}_\lambda,\hat{v}_\lambda,y_\lambda,0\}$. Then from (\ref{eq43}), (\ref{eq49}), (\ref{eq50}), (\ref{eq51}) and the Morse relation with $t=-1$ (see (\ref{eq1})), we have
	$$2(-1)^0 + (-1)^1 + (-1)^0 = (-1)^0,$$
	a contradiction. This means that there exists $\hat{y}_{\lambda}\in K_{\hat{\sigma}_\lambda}, \hat{y}_\lambda\notin\{\hat{u}_\lambda,\hat{v}_\lambda,y_\lambda,0\}$. Assuming without any loss of generality that the two constant sign solutions $\{\hat{u}_\lambda,\hat{v}_\lambda\}$ are extremal (that is, $\hat{u}_\lambda=u^*_\lambda, \hat{v}_\lambda=v^*_\lambda$, see Proposition \ref{prop13}), we have that
	$$\hat{y}_\lambda\in[\hat{v}_\lambda,\hat{u}_\lambda]\cap C^1_0(\overline\Omega)\ \mbox{(see Claim \ref{claim2} in the proof of Proposition \ref{prop14}) is nodal}.$$

	Moreover, as in the proof of Proposition \ref{prop14}, using Proposition \ref{prop3}, we show that
	$$\hat{y}_\lambda\in {\rm int}_{C^1_0(\overline{\Omega})}[\hat{v}_\lambda,\hat{u}_\lambda].$$
The proof of Theorem \ref{th16} is now complete.
\end{proof}

\section{Infinitely Many Nodal Solutions}
In this section we introduce a symmetry condition on $f(z,\cdot)$ (namely, that it is odd) and using Theorem \ref{th4}, we show that for all $\lambda>0$ big, problem \eqref{eqP} has a whole sequence of nodal solutions converging to zero in $C^1_0(\overline{\Omega})$.

The new hypotheses on the reaction term $f(z,x)$ are the following:

\smallskip
$H_3$: $f:\Omega\times\RR\rightarrow\RR$ is a Carath\'eodory function such that for almost all $z\in\Omega$, $f(z,0)=0, f(z,\cdot)$ is odd and hypotheses $H_3$ (i), (ii), (iii), (iv) are the same as the corresponding hypotheses $H_1$ (i), (ii), (iii), (iv).
\begin{theorem}\label{th17}
	If hypotheses $H_3$ hold, then we can find $\lambda^1_*>0$ such that for all $\lambda>\lambda^1_*$ problem \eqref{eqP} has a sequence of nodal solutions $\{u_n\}_{n\ge1}\subseteq C^1_0(\overline\Omega)$ such that $u_n\rightarrow0$ in $C_0^1(\overline\Omega)$.
\end{theorem}
\begin{proof}
	Let $\hat{f}(z,x)$ be as in Section 3. Consider its truncation at $\{-\eta,\eta\}$, that is, the Carath\'eodory function
	\begin{equation}\label{eq52}
		\tilde{f}(z,x)=\left\{\begin{array}{ll}
			f(z,-\eta) & \mbox{if}\ x<-\eta \\
			f(z,x) & \mbox{if}\ -\eta\le x\le\eta\ \mbox{(see (\ref{eq3}))} \\
			f(z,\eta) & \mbox{if}\ \eta<x.
		\end{array}
		\right.
	\end{equation}
	
	We set $\tilde{F}(z,x)=\int_0^x\tilde{f}(z,s)ds$ and consider the $C^1$-functional $\tilde{\varphi}_\lambda:W^{1,p}_0(\Omega)\rightarrow\RR$ defined by
	$$\tilde{\varphi}_\lambda(u)=\frac{1}{p}||Du||^p_p + \frac{1}{2}||Du||^2_2 - \lambda\int_\Omega\tilde{F}(z,u)dz\ \mbox{for all}\ u\in W^{1,p}_0(\Omega).$$
	
	Evidently $\tilde{\varphi}_\lambda$ is even, coercive (see (\ref{eq52})), hence it is also bounded below and satisfies the PS-condition. Moreover, $\tilde{\varphi}_\lambda(0)=0$.
	
	Let $Y\subseteq W^{1,p}_0(\Omega)$ be a finite dimensional subspace. All norms on Y are equivalent. So, we can find $\rho_0>0$ such that
	\begin{equation}\label{eq53}
		u\in Y,\ ||u||\le \rho_0 \Rightarrow |u(z)|\le\eta\ \mbox{for almost all}\ z\in\Omega.
	\end{equation}
	
	Hypothesis $H_3(ii)$ implies that we can find $\xi_1>0$ such that
	\begin{equation}\label{eq54}
		\tilde{F}(z,x)\ge\xi_1|x|^\beta\ \mbox{for almost all}\ z\in\Omega,\ \mbox{all}\ |x|\le\eta\ \mbox{(see (\ref{eq2}))}.
	\end{equation}
	
	Using (\ref{eq53}), (\ref{eq54}) and reasoning as in the proof of Proposition \ref{prop14}, we can find $\lambda^1_*>0$ such that for all $\lambda>\lambda^1_*$ we can find $\rho_\lambda>0$ for which we have
	$$\sup\left[\tilde{\varphi}_\lambda(u): u\in Y, ||u||=\rho_\lambda\right]<0.$$
	
	Applying Theorem \ref{th4}, we can find $\{u_n\}_{n\ge1}\subseteq K_{\tilde{\varphi}_\lambda}$ such that
	\begin{equation}\label{eq55}
		u_n\rightarrow0\ \mbox{in}\ W^{1,p}_0(\Omega).
	\end{equation}
	As before, using the nonlinear regularity theory and (\ref{eq54}), we have
	\begin{eqnarray*}
		& & u_n\rightarrow0\ \mbox{in}\ C^1_0(\overline\Omega), \\
		& \Rightarrow & u_n\in[v^*_\lambda,u^*_\lambda]\ \mbox{for\ all}\ n\ge n_0\ (\mbox{see Proposition (\ref{prop13})}), \\
		& \Rightarrow & \{u_n\}_{n\ge n_0}\ \mbox{are nodal solutions of \eqref{eqP}}\ \mbox{for}\ \lambda>\lambda^1_*.
	\end{eqnarray*}
The proof of Theorem \ref{th17} is now complete.
\end{proof}

\medskip
{\bf Acknowledgements.} This research was supported by the Slovenian Research Agency grants
P1-0292, J1-8131, J1-7025, and N1-0064. V.D.~R\u adulescu acknowledges the support through a grant of the Romanian Ministry of Research and Innovation, CNCS--UEFISCDI, project number PN-III-P4-ID-PCE-2016-0130,
within PNCDI III.

\end{document}